\documentclass[runningheads,a4paper]{llncs}

\usepackage{amssymb}
\usepackage{amscd}
\usepackage{amsmath}
\usepackage{mathrsfs}
\usepackage{graphicx}
\usepackage{url}
\usepackage{epstopdf}

\newcommand{\keywords}[1]{\par\addvspace\baselineskip\noindent\keywordname\enspace\ignorespaces#1}

\DeclareMathAlphabet\gothic{U}{euf}{m}{n}
\def\d{\mathrm{d}}
\newcommand{\vect}[1]{\boldsymbol{#1}}
\newcommand{\ul}[1]{\boldsymbol{#1}}

\newcommand{\R}{\mathbb{R}}
\newcommand{\C}{\mathbb{C}}
\newcommand{\LL}{\mathbb{L}}

\newcommand{\cR} {\mathcal{R}}
\newcommand{\cV} {\mathcal{V}}
\newcommand{\vx}{\mathbf{x}}
\newcommand{\ve}{\mathbf{e}}
\newcommand{\vy}{\mathbf{y}}
\newcommand{\vn}{\mathbf{n}}
\newcommand{\mR}{\mathbf{R}}
\newcommand{\Gh}{\gothic{h}}
\newcommand{\cA}{\mathcal{A}}
\newcommand{\cW}{\mathcal{W}}
\newcommand{\sgn}{\operatorname{sgn}}

\begin{document}

\mainmatter

%
%

\title{Invertible Orientation Scores of 3D Images}

\author{Michiel Janssen$^{1}$ %
\and Remco Duits$^{1,2}$
\and Marcel Breeuwer$^{2}$}
%
%

\titlerunning{Invertible Orientation Scores of 3D Images}
\authorrunning{Invertible Orientation Scores of 3D Images}
\institute{
Eindhoven University of Technology, The Netherlands, \\
$^1$Department of Mathematics and Computer Science,
\\$^2$Department of Biomedical Engineering.
 \\ \email{M.H.J.Janssen@tue.nl,R.Duits@tue.nl,
M.Breeuwer@tue.nl}}

%

%

\maketitle
\begin{abstract}
The enhancement and detection of elongated structures in noisy image data is relevant for many biomedical applications. To handle complex crossing structures in 2D images, 2D orientation scores $U: \mathbb{R} ^ 2\times S ^ 1 \rightarrow \mathbb{R}$ were introduced, which already showed their use in a variety of applications. Here we extend this work to 3D orientation scores $U: \mathbb{R} ^ 3 \times S ^ 2\rightarrow \mathbb{R}$. First, we construct the orientation score from a given dataset, which is achieved by an invertible coherent state type of transform. For this transformation we introduce 3D versions of the 2D cake-wavelets, which are complex wavelets that can simultaneously detect oriented structures and oriented edges. For efficient implementation of the different steps in the wavelet creation we use a spherical harmonic transform. Finally, we show some first results of practical applications of 3D orientation scores.

\keywords{Orientation Scores, Reproducing Kernel Spaces, 3D Wavelet Design, Scale Spaces on SE(3), Coherence Enhancing Diffusion on SE(3)}
\end{abstract}

%
%

\section{Introduction}
The enhancement and detection of elongated structures is important in many biomedical image analysis applications. These tasks become problematic when multiple elongated structures cross or touch each other in the data. In these cases it is useful to decompose an image in local orientations by constructing an orientation score. In the orientation score, we extend the domain of the data to include orientation in order to separate the crossing or touching structures (Fig. \ref{fig:2DOS}). From 3D data $f:\R^ 3\rightarrow \R$ we construct a 3D orientation score $U:\R^ 3\times S ^ 2 \rightarrow \R$, in a similar way as is done for the more common case of 2D data $f:\R^ 2\rightarrow \R$ and 2D orientation score $U:\R^ 2\times S^1 \rightarrow \R$. Next, we consider operations on orientation scores, and process our data via orientation scores (Fig. \ref{fig:OverviewOperations}). For such operations it is important that the orientation score transform is invertible, in a well-posed manner. In comparison to continuous wavelet transforms on the group of 3D rotations, translations and scalings, we use all scales simultaneously and exclude the scaling group from the wavelet transform and its adjoint, yielding a coherent state type of transform \cite{Ali1998}, see App.A. This makes it harder to design appropriate wavelets, but has the computational advantage of only needing a single scale transformation.

The 2D orientation scores have already showed their use in a variety of applications. In \cite {Franken2009,Sharma2014} the orientation scores were used to perform crossing-preserving coherence-enhancing diffusions. These diffusions greatly reduce the noise in the data, while preserving the elongated crossing structures. Next to these generic enhancement techniques, the orientation scores also showed their use in retinal vessel segmentation \cite {Bekkers2013}, where they were used to better handle crossing vessels in the segmentation procedure.

To perform detection and enhancement operations on the orientation score, we first need to transform a given greyscale image or 3D dataset to an orientation score in an invertible way. In previous works various wavelets were introduced to perform a 2D orientation score transform. Some of these wavelets did not allow for an invertible transformation (e.g. Gabor wavelets \cite {Lee1996}). A wavelet that allows an invertible transformation was proposed by Kalitzin \cite {Kalitzin1999}. A generalization of these wavelets was found by Duits \cite {ThesisDuits} who derived a unitarity result and expressed the wavelets in a basis of eigenfunctions of the harmonic oscillator. This type of wavelet was also extended to 3D. This wavelet however has some unwanted properties such as poor spatial localization (oscillations) and the fact that the maximum of the wavelet did not lie at its center \cite[Fig. 4.11]{ThesisDuits}. In \cite {ThesisDuits} a class of cake-wavelets were introduced, that have a cake-piece shaped form in the Fourier domain (Fig. \ref{fig:CakeWavelets}). The cake-wavelets simultaneously detect oriented structures and oriented edges by constructing a complex orientation score $U:\R ^ 2 \times S^1 \rightarrow \C$. Because the different cake-wavelets cover the full Fourier spectrum, invertibility is guaranteed.

In this paper we propose an extension of the 2D cake-wavelets to 3D. First, we discuss the theory of invertible orientation score transforms. Then we construct 3D cake-wavelets and give an efficient implementation using a spherical harmonic transform. Finally we mention two application areas for 3D orientation scores and show some preliminary results for both of them. In the first application, we present a practical proof of concept of a natural extension of the crossing preserving coherence enhancing diffusion on invertible orientation scores (CEDOS) \cite {Franken2009} to the 3D setting. Compared to the original idea of coherence enhancing diffusion acting directly on image-data \cite{Weickert1999,BurgethBook2009,Burgeth2012} we have the advantage of preserving crossings. Diffusions on SE(3) have been studied in previous SSVM-articles, see e.g. \cite {Creusen2012}, but the full generalization of CEDOS to 3D was never established.


\begin{figure}[b]
\parbox{.64\textwidth}{
	\includegraphics[width=.60\textwidth]{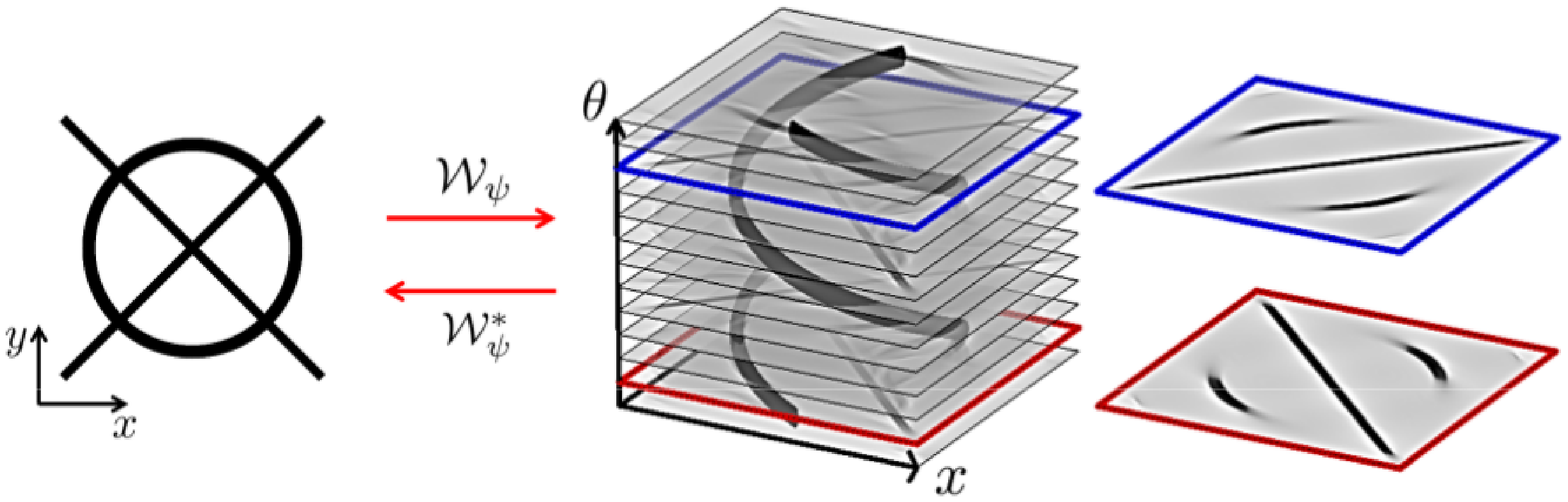}
}\hfill
\parbox{.35\textwidth}{
	\includegraphics[width=.35\textwidth]{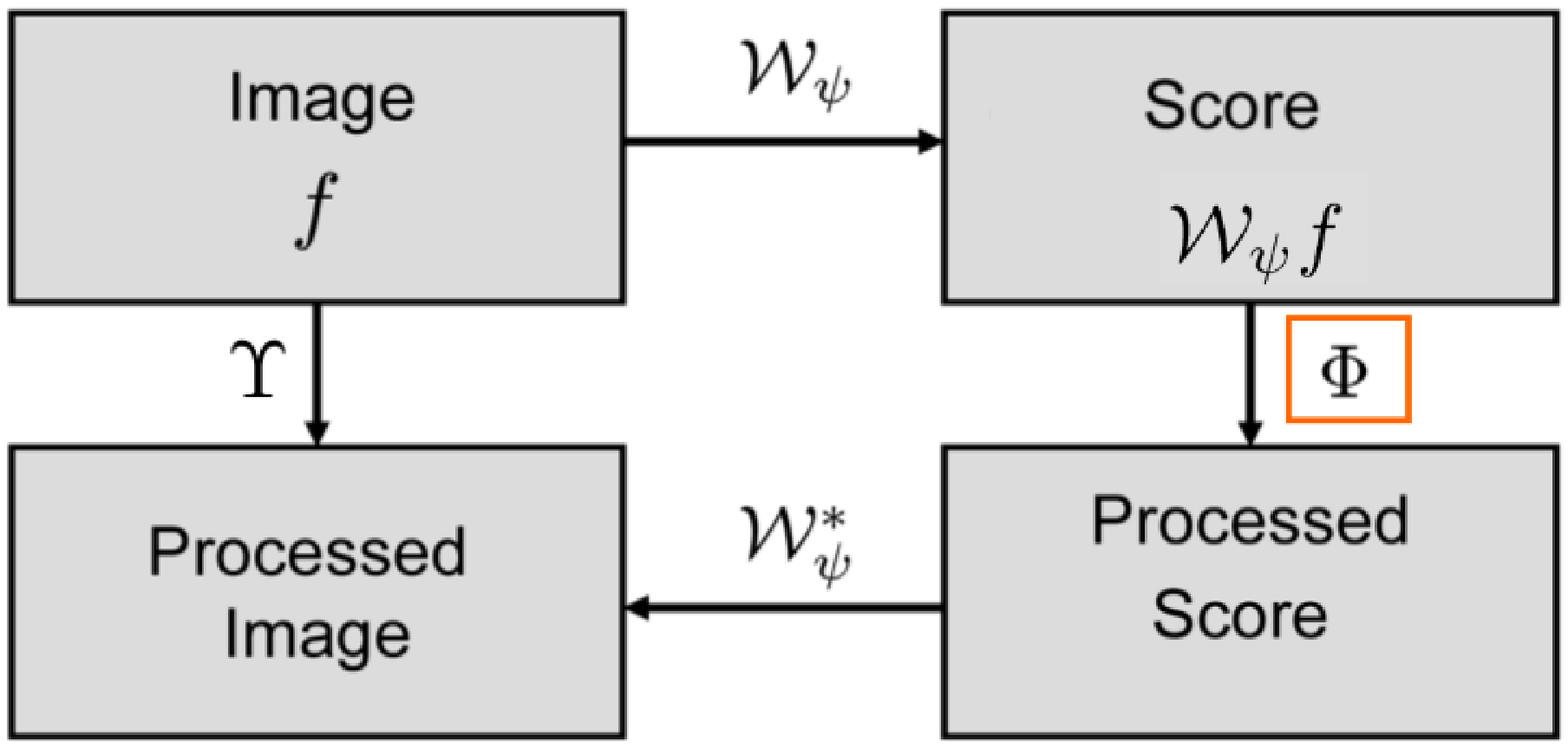}
}\hfill\null
\parbox[t]{.60\linewidth}{
	\caption{2D Orientation score for an exemplary image. In the orientation score crossing structures are disentangled because the different structures have a different orientation.}
	\label{fig:2DOS}
}\hfill
\parbox[t]{.35\linewidth}{
	\caption{A schematic view of image processing via invertible orientation scores.}%
	\label{fig:OverviewOperations}
}\hfill\null
\end{figure}


\begin{figure}[t]%
\centering
\includegraphics[width=0.6\textwidth]{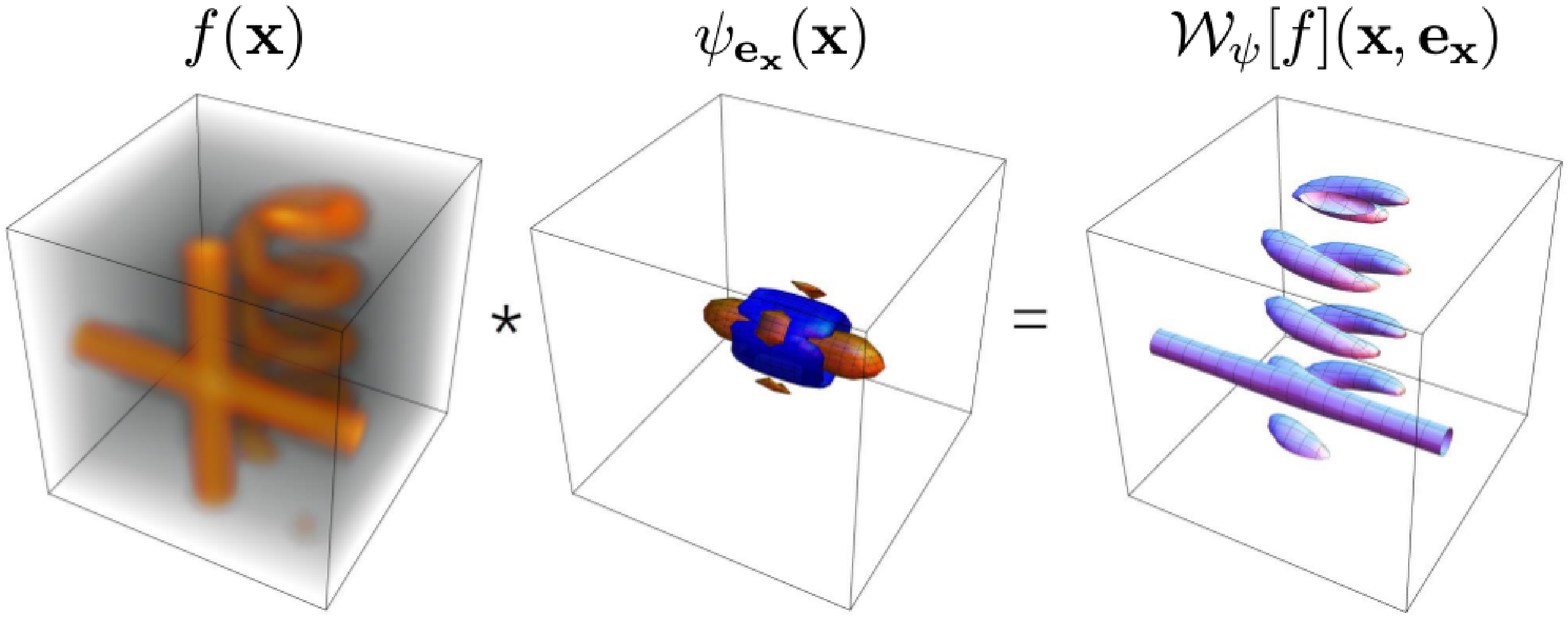}\\
\includegraphics[width=0.7\textwidth]{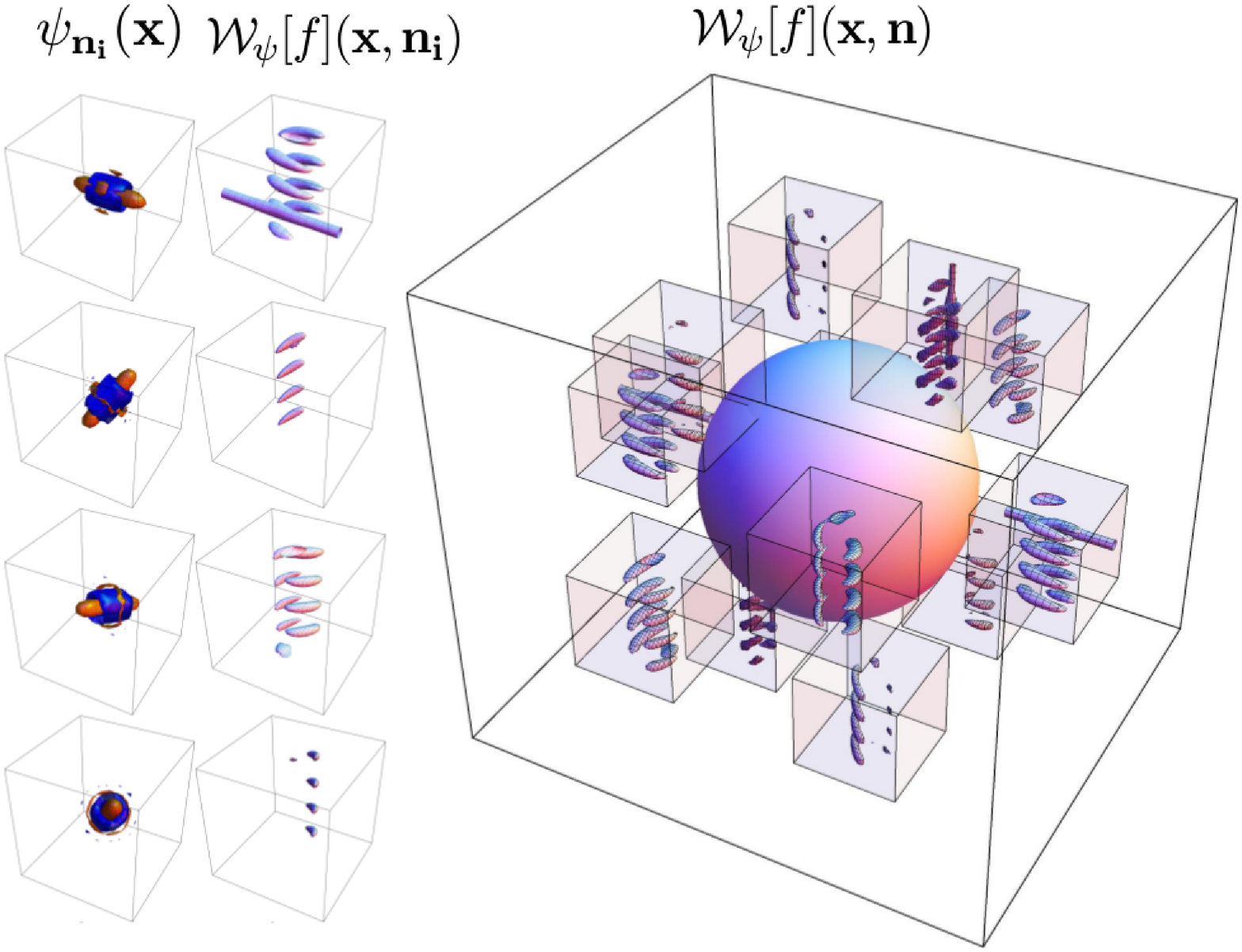}
\caption{Creating a 3D orientation score. Top: The data $f$ is correlated with an oriented filter $\psi_{\ve_x}$ to detect structures aligned with the filter orientation $\ve_x$. Bottom left: This is repeated for a discrete set of filters with different orientations. Bottom right: The collection of 3D datasets constructed by correlation with the different filters is an orientation score and is visualized by placing a 3D dataset on a number of orientations.}%
\label{fig:3DOSIntro}%
\end{figure}

\section {Invertible Orientation Scores}
An invertible orientation score $\cW_\psi[f]:\R^3 \times S^2 \rightarrow \C$ is constructed from a given ball-limited 3D dataset $f \in \LL_2^{\varrho} (\R^3)=\{f \in \LL_2 (\R^3)| \textrm{supp}(\mathcal{F}f) \subset B_{0,\varrho}\}$, with $\varrho>0$ by correlation $\star$ with an anisotropic kernel
\begin{equation}
(\cW_\psi[f])(\vx,\vn)=(\overline{\psi_\vn} \star f)(\vx)=\int_{\R ^ 3}  \overline {\psi_\vn (\vx'-\vx)}f(\vx')\, \d \vx ',
\label{eq:Construction1}
\end{equation}
where $\psi \in \LL_2(\R^3) \cap \LL_1(\R^3)$ is a wavelet aligned with and rotationally symmetric around the $z$-axis, and $\psi_{\vn} (\vx)=\psi (\mR_{\vn}^T \vx)  \in \LL_2(\R^3)$ the rotated wavelet aligned with $\vn$. Here $\mR_{\vn}$ is any rotation which rotates the $z$-axis onto $\vn$ where the specific choice of rotation does not matter because of the rotational symmetry of $\psi$. The overline denotes a complex conjugate. The exact reconstruction formula for this transformation is
\begin{equation}
\begin{split}
f(\vx) &= (\cW_\psi^{-1}[\cW_\psi[f]])(\vx) \\
&= \mathcal{F}_{\R^3}^{-1} \left[ M_\psi^{-1}  \mathcal{F}_{\R^3} \left[ \tilde{\vx} \mapsto \int_{S^2} (\check {\psi}_\vn \star \cW_\psi[f](\cdot,\vn))(\tilde{\vx})\, \d \sigma(\vn) \right] \right] (\vx),
\end{split}
\label{eq:Reconstruction1}
\end{equation}
 with $\mathcal{F}_{\R^3}$ the Fourier transform on $\R^3$ given by $(\mathcal{F} f)(\vect{\omega}) = (2\pi) ^ {-\frac{3}{2}} \int_{\R ^ 3 } e ^ {-i \vect{\omega}\cdot \vx} f(\vx) \d \vx$ and $\check {\psi}_\vn (\vx ) =\psi _\vn(-\vx)$. In fact $\cW_\psi$ is a unitary mapping on to a reproducing kernel space, see App. A. The function $M_\psi: \R^3 \rightarrow \R^+ $ is given by
\begin{equation}
M_\psi(\vect{\omega}) = (2\pi) ^ {\frac{3}{2}}  \int_{S^2} \left| \mathcal{F}_{\R^3} [\psi_\vn](\vect{\omega}) \right|^2 \d \sigma(\vn).
\end{equation}
The function $M_\psi$ quantifies the stability of the inverse transformation \cite{ThesisDuits}, since $M_\psi(\vect{\omega})$ specifies how well frequency component $\vect{\omega}$ is preserved by the cascade of construction and reconstruction when $M_\psi^{-1}$ would not be included in Eq.~\!(\ref{eq:Reconstruction1}). An exact reconstruction is possible as long as
\begin{equation}
\exists_{M> 0,\delta> 0} \quad 0<\delta \leq M_\psi (\vect{\omega}) \leq M<\infty, \quad \textrm {for all } \vect{\omega}=B_{0,\varrho}.
\label{eq:AdmissibilityRequirement}
\end{equation}
In practice it is best to aim for $M_\psi (\vect{\omega}) \approx 1,$  in view of the condition number of $W_\psi:\LL_2^\varrho (\R^3)\rightarrow \LL_2^\varrho (\R^3\times S^2)$ with $W_\psi f=\cW_\psi f$. Also, when  $M_\psi (\vect{\omega}) = 1$  we have $\LL_2$-norm preservation
\begin{equation}
\|f\|_{\LL_2(\R^3)}^2 = \|\cW_\psi f\|_{\LL_2 (\R^3\times S^2)}^2,\quad \textrm{for all } f \in \LL_2^\varrho (\R^3),
\end{equation}
and Eq.~\!(\ref{eq:Reconstruction1}) simplifies to $f(\vx) = \int_{S^2} (\check{\psi}_\vn \star \cW_\psi[f](\cdot,\vn))(\vx) \d \sigma(\vn)$. We can further simplify the reconstruction for wavelets for which $(2\pi) ^ {\frac{3}{2}}  \int_{S^2} \mathcal{F}_{\R^3} [\psi_\vn](\vect{\omega}) \d \sigma(\vn) \approx 1$, where the reconstruction formula simplifies to an integration over orientations

\begin{equation}
f(\vx) \approx \int_{S^2} \mathcal{W}_{\psi}f(\vx,\vn)\, \d \sigma(\vn).
\label{eq:Reconstruction2Approximation}
\end{equation}








\subsection {Discrete Invertible Orientation Score Transformation}
In the previous section, we considered a continuous orientation score transformation. In practice, we have only a finite number of orientations. To determine this discrete set of orientations we uniformly sample the sphere using platonic solids and/or refine this using tessellations of the platonic solids.

Assume we have a number $N_o$ of orientations $\cV =\{\vn_1,\vn_2,...,\vn_{N_o}\}\subset S^2$, and define the discrete invertible orientation score $\cW_\psi^d[f]:\R^3\times \cV \rightarrow \C$ by
\begin{equation}
(\cW_\psi^d[f])(\vx,\vn_i)=(\overline{\psi_{\vn_i}} \star f)(\vx).
\label{eq:construction1Discrete}
\end{equation}
The exact reconstruction formula is in the discrete setting given by
\begin{equation}
\begin{split}
f(\vx) &= ((\cW_\psi^d)^{-1}[\cW_\psi^d[f]])(\vx) \\
&= \mathcal{F}_{\R^3}^{-1} \left[ (M_\psi^d)^{-1}  \mathcal{F}_{\R^3} \left[ \tilde {\vx} \rightarrow \sum_{i=1}^{N_o} (\check {\psi}_{\vn_{i}} \star \cW_\psi^d[f](\cdot,\vn_i))(\tilde {\vx}) \d \sigma(\vn_i) \right] \right] (\vx),
\end{split}
\label{eq:Reconstruction1Discrete}
\end{equation}
with $\d \sigma(\vn_i)$ the discrete spherical area measure which for reasonably uniform spherical sampling can be approximated by $\d \sigma(\vn_i)\approx \frac{4 \pi}{N_o}$, and
\begin{equation}
M_\psi^d(\vect{\omega}) = (2\pi) ^ {\frac{3}{2}}  \sum_{i=1}^{N_o} \left| \mathcal{F}_{\R^3} [\psi_{\vn_i}](\vect{\omega}) \right|^2 \d \sigma(\vn_i).
\end{equation}
Again, an exact reconstruction is possible iff $0<\delta\leq M_\psi^d (\vect{\omega})\leq M<\infty$.

\section{3D Cake-Wavelets}
A class of 2D cake-wavelets, see \cite{ThesisDuits}, was successfully used for the 2D orientation score transformation. We now generalize these 2D cake-wavelets to 3D cake-wavelets. Our 3D transformation using the 3D cake-wavelets should fulfill a set of requirements, compare \cite{Franken2009} :

\begin{enumerate}
	\item The orientation score should be constructed for a finite number ($N_o$) of orientations.
	\item The transformation should be invertible and all frequencies should be transferred equally to the orientation score domain ($M_\psi^d \approx 1$).
	\item The kernel should be strongly directional.
\item The kernel should be polar separable in the Fourier domain, i.e., $(\mathcal{F}\psi) (\vect{\omega}) =g(\rho) h (\theta,\phi)$, with $\vect{\omega} = (\omega_x,\omega_y,\omega_z) = (\rho \sin \theta \cos \phi,\rho \sin \theta \sin \phi,\rho \cos \theta)$. Because by definition the wavelet $\psi$ has rotational symmetry around the $z$-axis we have $h (\theta,\phi) = \Gh (\theta)$.
\item The kernel should be localized in the spatial domain, since we want to pick up local oriented structures.
\item The real part of the kernel should detect oriented structures and the imaginary part should detect oriented edges. The constructed oriented score is therefore a complex orientation score.
\end{enumerate}

\subsection {Construction of Line and Edge Detectors}\label{ssect:Cake}
We now discuss the procedure used to make 3D cake-wavelets. According to requirement 4 we only consider polar separable wavelets in the Fourier domain, so that $(\mathcal{F}\psi) (\vect{\omega}) =g(\rho) \Gh (\theta)$. For the radial function $g (\rho)$ we use, as in \cite{Franken2009},
\begin{equation}
g (\rho)=\mathcal{M}_N(\rho^2 t^{-1})= e^{-\frac{\rho^2}{t}}\sum_{k=0}^N \frac{(\rho^2 t^{-1})^k}{k !},
\label{eq:}
\end{equation}
which is a Gaussian function with scale parameter $t$ multiplied by the Taylor approximation of its reciprocal to order $N$ to ensure a slower decay. This function should go to 0 when $\rho$ tends to the Nyquist frequency $\rho_N$. Therefore the inflection point of this function is fixed at $\gamma \, \rho_N$ with $0 \ll \gamma < 1$ by setting $t=\frac{2 (\gamma \,\rho_N)^ 2}{1+ 2N}$. In practice we have $\varrho=\rho_N$, and because radial function $g$ causes $M_\psi^d$ to become really small when coming close to the Nyquist frequency, reconstruction Eq.\eqref{eq:Reconstruction1Discrete} becomes unstable. We solve this by either using approximate reconstruction Eq.\eqref{eq:Reconstruction2Approximation} or by replacing $M_\psi^d \rightarrow \max(M_\psi^d,\epsilon)$, with $\epsilon$ small. Both make the reconstruction stable at the cost of not completely reconstructing the highest frequencies which causes some additional blurring.

We now need to find an appropriate angular part $\Gh$ for the cake-wavelets. First, we specify an orientation distribution $A:S^2\rightarrow \R^+$, which determines what orientations the wavelet should measure. To satisfy requirement 3 this function should be a localized spherical window, for which we propose a B-spline $A(\theta,\phi) =B ^k (\frac{\theta} {s_\theta})$, with $s_{\theta}>0$ and $B^k$ the $k$th order B-spline given by
\begin{equation}
 B ^k (x) = (B ^ {k-1}*B ^ 0)(x) , \quad B ^ 0(x) =\begin {cases} 1 & \text {if} -\frac{1}{2} < x < \frac{1}{2} \\
 0 & \text{otherwise} \end{cases}.
\label{eq:}
\end{equation}
The parameter $s_{\theta}$ determines the trade-off between requirements 2 and 3, where higher values give a more uniform $M_\psi^d$ at the cost of less directionality.

First consider setting $h=A$ so that $\psi$ has compact support within a convex cone in the Fourier domain. The real part of the corresponding wavelet would however be a plate detector and not a line detector (Fig. \ref{fig:cakePieceCreatesPlateDetector}). The imaginary part is already an oriented edge detector, and so we set

\begin{equation}
\Gh_ {Im} (\phi) = A(\theta,\phi) -A(\pi -\theta,\phi+\pi)=B ^k \left(\frac{\theta} {s_\theta}\right)-B ^k \left(\frac{\pi-\theta} {s_{\theta}}\right),
\label{eq:Antisymmetrize}
\end{equation}
where the real part of the earlier found wavelet vanishes by anti-symmetrization of the orientation distribution $A$ while the imaginary part remains. As to the construction of $h_ {Re}$, there is the general observation that we detect a structure that is perpendicular to the shape in the Fourier domain, so for line detection we should aim for a plane detector in the Fourier domain. To achieve this we apply the Funk transform to $A$, and we define
\begin{equation}
h_{Re}(\theta,\phi) =F A(\theta,\phi) = \int_{S_p(\vn (\theta,\phi))} \! A(\vn') \, \d s(\vn'),
\label{eq:CakeWaveletRe}
\end{equation}
where integration is performed over $S_p (\vn)$ denoting the great circle perpendicular to $\vn (\theta,\phi) = (\sin \theta \cos \phi,\sin \theta \sin \phi,\cos \theta)$. This transformation preserves the symmetry of $A$, so we have $h_{Re} (\theta,\phi) =\Gh_{Re} (\theta)$. Thus, we finally set
\begin{equation}
\Gh(\theta) =\Gh_{Re}(\theta) +\Gh_{Im} (\theta).
\label{eq:}
\end{equation}
For an overview of the transformations see  Fig. \ref{fig:CakeWavelets}.

\subsection {Efficient Implementations Via Spherical Harmonics}
In Subsection \ref{ssect:Cake} we defined the real part and the imaginary part of the wavelets in terms of a given orientation distribution. In order to efficiently implement the various transformations (e.g. Funk transform), and to create the various rotated versions of the wavelet we express our orientation distribution $A$ in a spherical harmonic basis $\{Y_l^m\}$ up to order $L$:
\begin{equation}
A(\theta,\phi) =\sum_{l = 0} ^L \sum_{m= -l} ^ {l} c_{l,m} Y_l ^m(\theta,\phi), \quad L \in \mathbb{N}.
\label{eq:}
\end{equation}
Because of the rotational symmetry around the $z$-axis, we only need the spherical harmonics with $m = 0$, i.e., $A(\theta,\phi) =\sum_{l = 0} ^L  c_{l, 0} Y_l ^ 0 (\theta,\phi)$. For determining the spherical harmonic coefficients we use the pseudo-inverse of the discretized inverse spherical harmonic transform (see \cite[Section 7.1]{DuitsIJCV2010}), with discrete orientations given by an icosahedron of tesselation order 15.

\subsubsection {Funk Transform}
According to \cite {Descoteaux2007}, the Funk transform of a spherical harmonic equals
\begin{equation}
F Y_l^m  (\theta,\phi) = \int_{S_p(\vn(\theta,\phi))} \! Y_l^m(\vn') \, \d s(\vn') =2\pi P_l(0)Y_l^m  (\theta,\phi),
\label{eq:}
\end{equation}
with  $P_l (0) $ the Legendre polynomial of degree $l$ evaluated at $0$. We can therefore apply the Funk transform to a function expressed in a spherical harmonic basis by a simple transformation of the coefficients $c_l^m \rightarrow  2\pi P_l(0)c_l^m$.

\subsubsection {Anti-Symmetrization}
We have $Y_l^m(\pi -\theta,\phi +\pi)=(-1)^l Y_l^m(\theta,\phi)$. We therefore anti-symmetrize the orientation distribution Eq.~\!(\ref{eq:Antisymmetrize}) via $c_l^m\rightarrow (1-(-1)^l) c_l ^ m$.

\subsubsection {Making Rotated Wavelets}
To make the rotated versions $\psi_\vn$ of wavelet $\psi$ we have to find $h_\vn$ in $\Psi_\vn =g (\rho)h_\vn (\theta,\phi)$. To achieve this we use the steerability of the spherical harmonic basis. Spherical harmonics rotate according to the irreducible representations of the SO(3) group $D_{m,m '} ^ l (\alpha,\beta,\gamma)$ (Wigner-D functions)

\begin{equation}
\cR_ {\mR_{\alpha ,\beta ,\gamma}}Y_l^m(\theta ,\phi )=\sum _{m'=l}^l D_{m,m'}^l(\alpha ,\beta ,\gamma )Y_l^{m'}(\theta ,\phi ).
\label{eq:}
\end{equation}
Here $\alpha, \beta $ and $\gamma$ denote the Euler angles with counterclockwise rotations, i.e., $\mR =\mR_ {\ve_z,\alpha} \mR_ {\ve_y,\beta} \mR_ {\ve_z,\gamma}$. This gives
\begin{equation}
h_\vn (\theta,\phi)=\cR_{\mR_{\alpha ,\beta ,\gamma }} h(\theta ,\phi ) = \overset{L }{\sum _{l=0} }\overset{l}{\sum _{m=-l} }\overset{l}{\sum _{m'=-l} }a_{l, m} D_{m, m'}^l(\alpha ,\beta ,\gamma )Y_l^{m'}(\theta ,\phi ).
\label{eq:Rotationh}
\end{equation}
Because both anti-symmetrization and Funk transform preserve the rotational symmetry of $A$, we have $h(\theta,\phi) =\sum_{l = 0} ^L  a_{l, 0} Y_l ^ 0 (\theta,\phi)$, and Eq.~\!(\ref{eq:Rotationh}) reduces to
\begin{equation}
h_\vn (\theta,\phi)= \overset{L }{\sum _{l=0}} \overset{l}{\sum _{m'=-l} } a_{l, 0} D_{0, m'}^l(0 ,\beta ,\gamma )Y_l^{m'}(\theta ,\phi ) .
\label{eq:}
\end{equation}

\begin{figure}[p]
\centering
	\includegraphics[width=0.9 \textwidth]{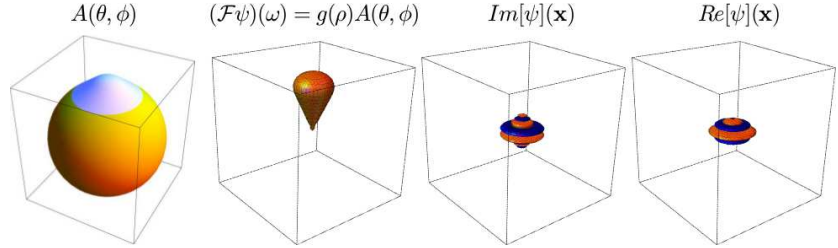}
  \caption{When directly setting orientation distribution $A$ as angular part of the wavelet $h$ we construct plate detectors. From left to right: Orientation distribution $A$, wavelet in the Fourier domain, the plate detector (real part) and the edge detector (imaginary part). Orange: Positive iso-contour. Blue: Negative iso-contour. Parameters used: $L=16,s_\theta=0.6,k=2,N=20,\gamma=0.8$ and evaluated on a grid of 51x51x51 pixels.}
	\label{fig:cakePieceCreatesPlateDetector}
\end{figure}

\begin{figure}[p]
\centering
\includegraphics[width=0.85 \textwidth]{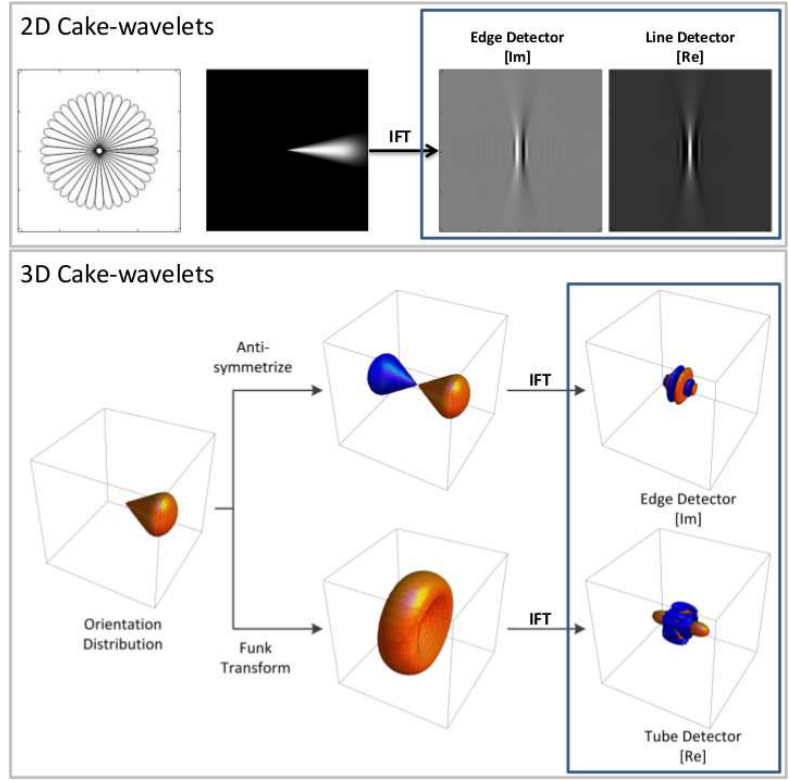}%
\caption{Cake-Wavelets. Top: 2D cake-wavelets. From left to right: Illustration of the Fourier domain coverage, the wavelet in the Fourier domain and the real and imaginary part of the wavelet in the spatial domain. \cite{Bekkers2013}. Bottom: 3D cake-wavelets. Overview of the transformations used to construct the wavelets from a given orientation distribution. Upper part: The wavelet according to Eq.~\!(\ref{eq:Antisymmetrize}). Lower part: The wavelet according to Eq.~\!(\ref{eq:CakeWaveletRe}). IFT: Inverse Fourier Transform. Parameters used: $L=16,s_\theta=1.05,k=2,N=20,\gamma=0.8$ and evaluated on a grid of 31x31x31 pixels.}%
\label{fig:CakeWavelets}%
\end{figure}

\section {Applications}

\subsection {Adaptive Crossing Preserving Flows}
We now use the invertible orientation score transformation to perform data-enhancement according to Fig. \ref{fig:OverviewOperations}. Because $\R ^ 3\times S ^ 2$ is not a Lie group, it is common practice to embed the space of positions and orientations in the Lie group of positions and rotations SE(3) by setting
\begin{equation}
\tilde{U}(\vx,\mR)=U(\vx,\mR \cdot \ve_z),\quad U(\vx,\vn)=\tilde{U}(\vx,\mR_n),
\label{eq:EquivalenceRelation}
\end{equation}
with $\mR_\vn$ any rotation for which $\mR_\vn \cdot \ve_z = \vn$. This holds in particular for orientation scores $U=\mathcal{W}_{\psi}f$.
The operations $\Phi $ which we consider are scale spaces on SE(3) (diffusions), and are given by $\Phi=\Phi_t$ with
\begin{equation}
\Phi_t(U)(\vy,\vn)=\tilde{W}(y,\mR_\vn,t).
\label{eq:}
\end{equation}
Here $\tilde {W}$ is the solution of
\begin{equation}
\quad \frac{\partial \tilde{W}}{\partial t}(g,t) =\sum_{i,j=1}^6 \cA_i|_g D_{i j} \cA_j|_g \tilde{W}(g,t),\quad \tilde{W}|_{t=0}=\widetilde{\cW_\psi[f]},
\label{eq:}
\end{equation}
where in coherence enhancing diffusion on orientation scores (CEDOS) $D_{i j}$ is adapted locally to data $\widetilde{\cW_\psi[f]}$ based on exponential curve fits (see \cite{GaugeFrameNew}), and with $\cA_i|_{g=(\vx,\mR)}=(L_g)_* \cA_i|_e$ the left-invariant vector fields on SE(3), for motivation and details see \cite {DuitsIJCV2010}. Furthermore $D_{i j}$ is chosen such that equivalence relation Eq.~\!\eqref{eq:EquivalenceRelation} is maintained for $\tilde {W}$. These operations are already used without adaptivity in the field of diffusion weighted MRI, where similar data (of the type $\R ^ 3\times S ^ 2\rightarrow \R ^ +$) is enhanced \cite {DuitsIJCV2010}. We then obtain Euclidean invariant image processing via
\begin{equation}
\Upsilon f = \cW_\psi ^ {*,ext}\circ \Phi \circ \cW_\psi f= \cW_\psi ^ {*}\circ \mathbb {P}_\psi  \Phi \circ \cW_\psi f
\label{eq:}
\end{equation}
which includes inherent projection $\mathbb{P}_\psi$ of orientation scores, even if $\Phi=\Phi_t$ maps outside of the space of orientation scores in the embedding space (see App. \ref{app:AppendixA}). Below we show some preliminary results of these flows that enhance the elongated structures while preserving the crossing, Fig. \ref{fig:Results} and Fig. \ref{fig:ResultsAdamKiewitzc}.

\begin{figure}[ht!]
\centering
  \includegraphics[width=0.24 \textwidth]{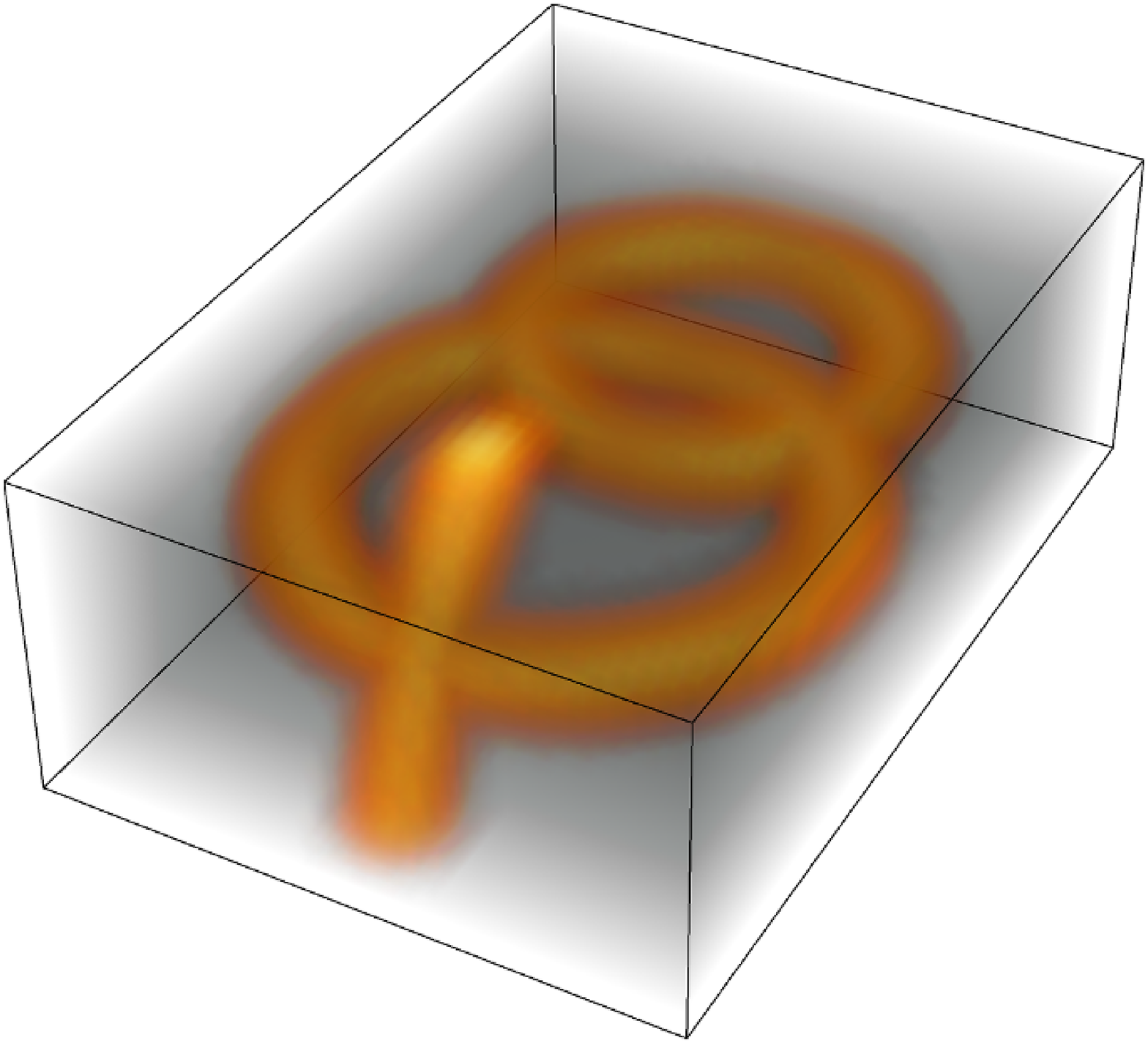}
	\includegraphics[width=0.16 \textwidth]{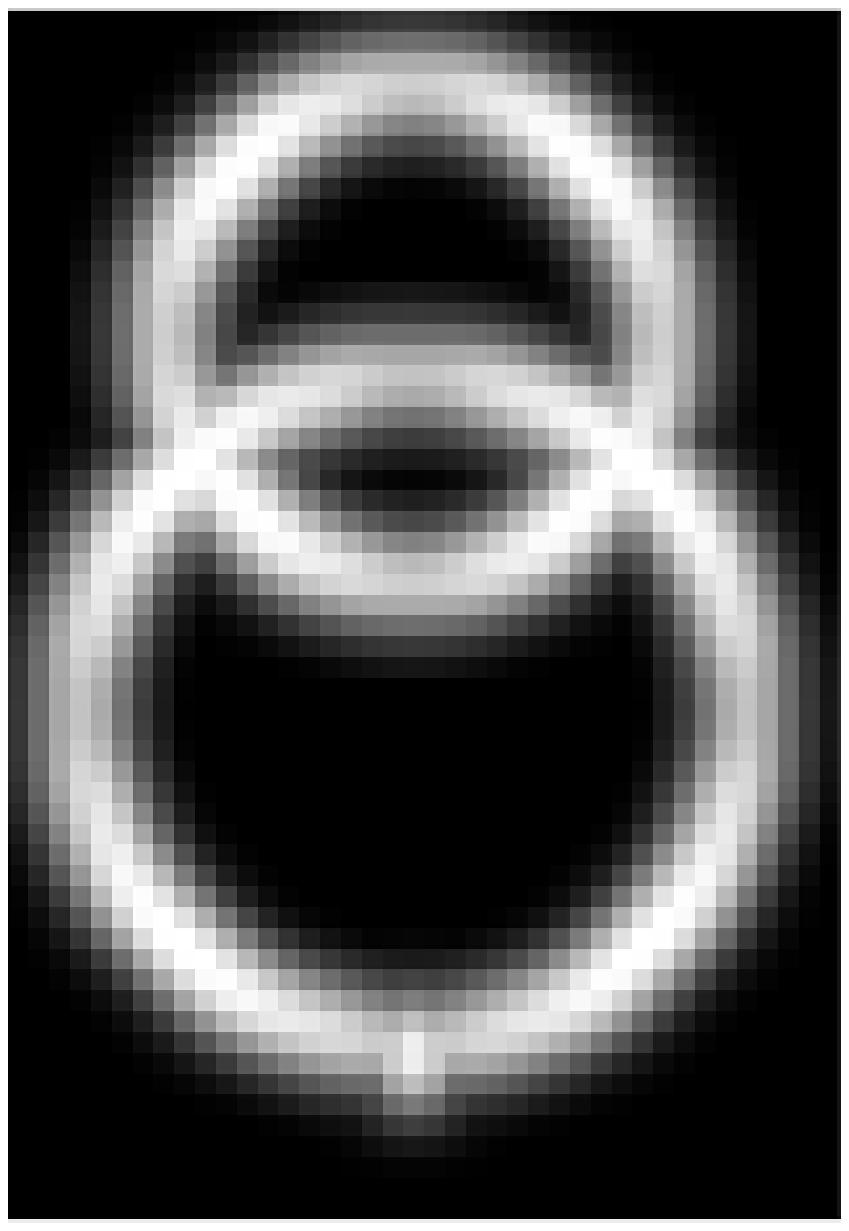}
	\includegraphics[width=0.16 \textwidth]{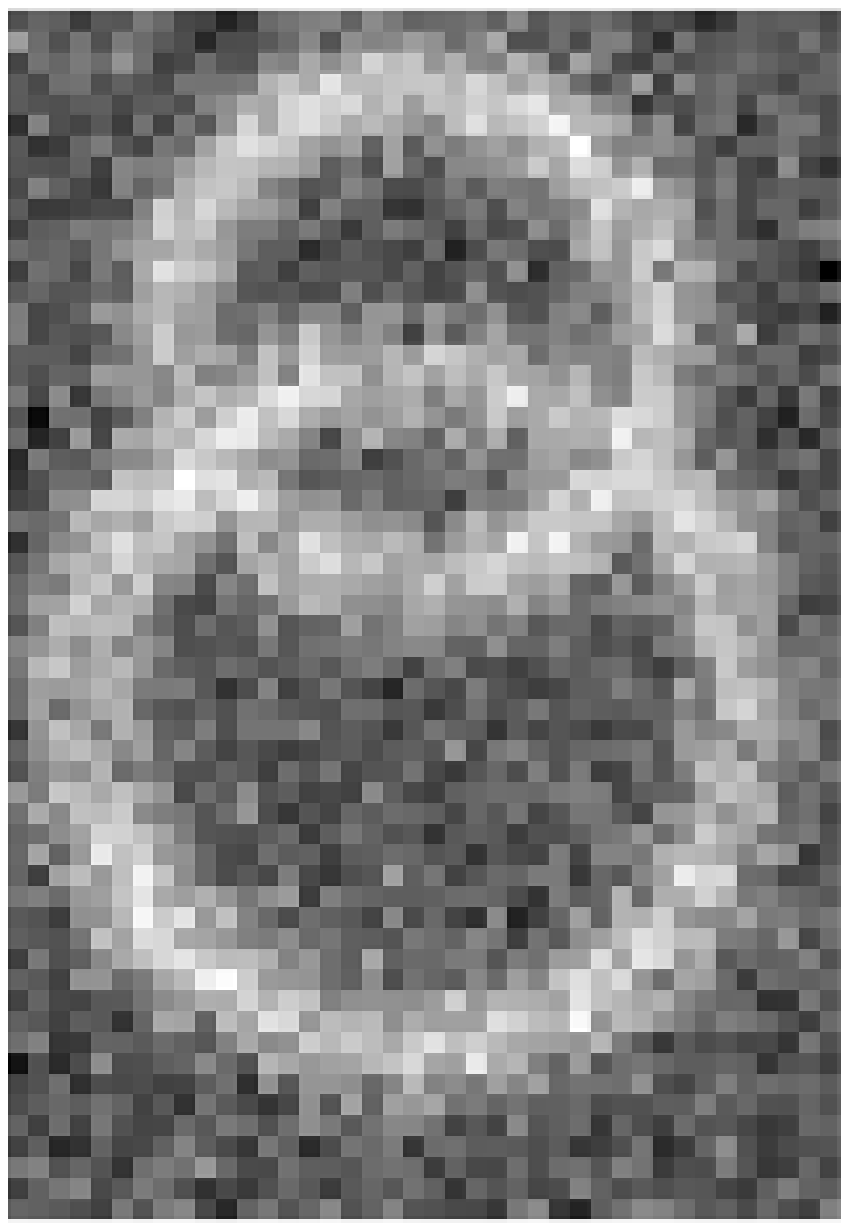}
	\includegraphics[width=0.16 \textwidth]{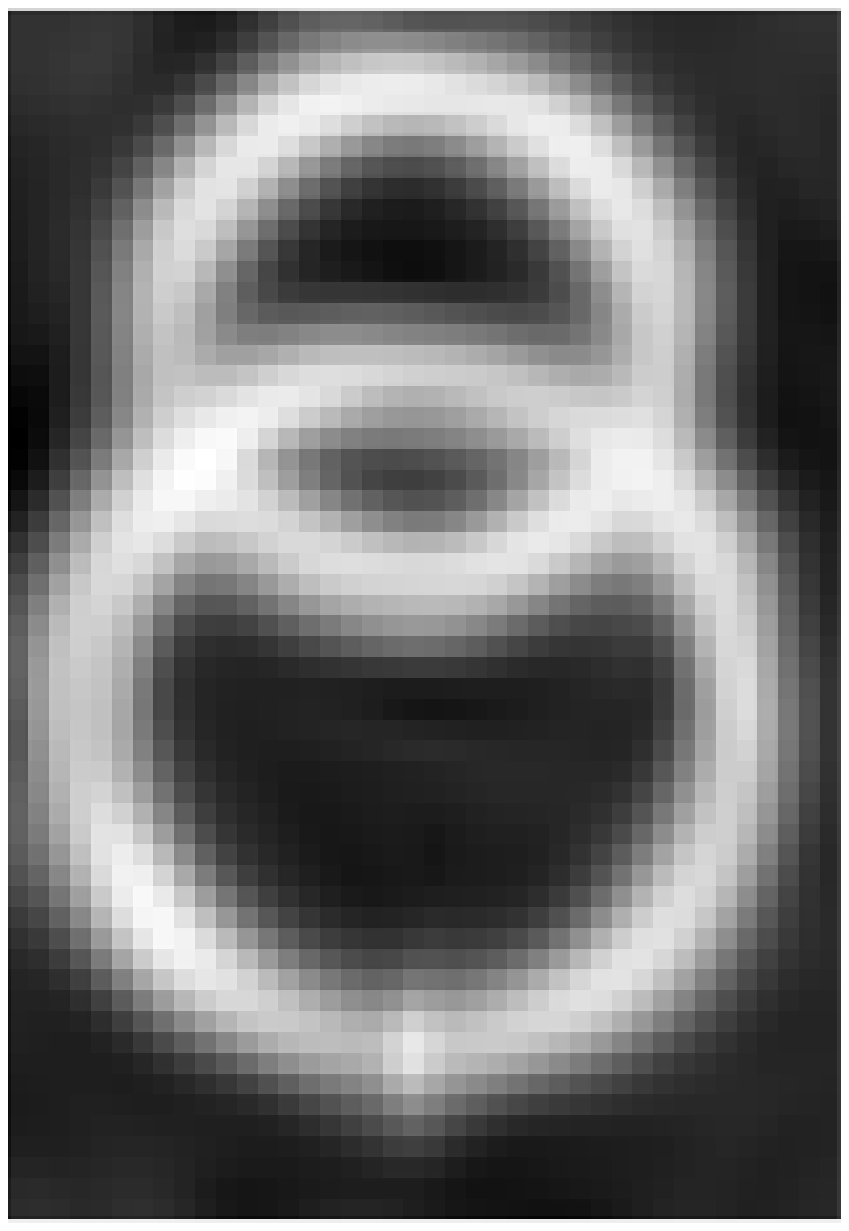}
  \caption{Adaptive Crossing Preserving Flows. From left to right 3D visualization of artificial data, slice of data, slice of (data + Gaussian noise), slice of enhanced data. For the orientation score transformation we use: $N_0=42,s_\theta=0.7,k=2,N=20,\gamma=0.85,L=16$ evaluated on a grid of 21x21x21 pixels. We use approximate reconstruction Eq.\eqref{eq:Reconstruction1Discrete}, and for diffusion we set $t=10$. For the choice of $D_{i j}$ in CEDOS, see \cite{GaugeFrameNew}.}
	\label{fig:Results}\end{figure}

\begin{figure}[ht!]
\centering
  \includegraphics[width=0.17 \textwidth]{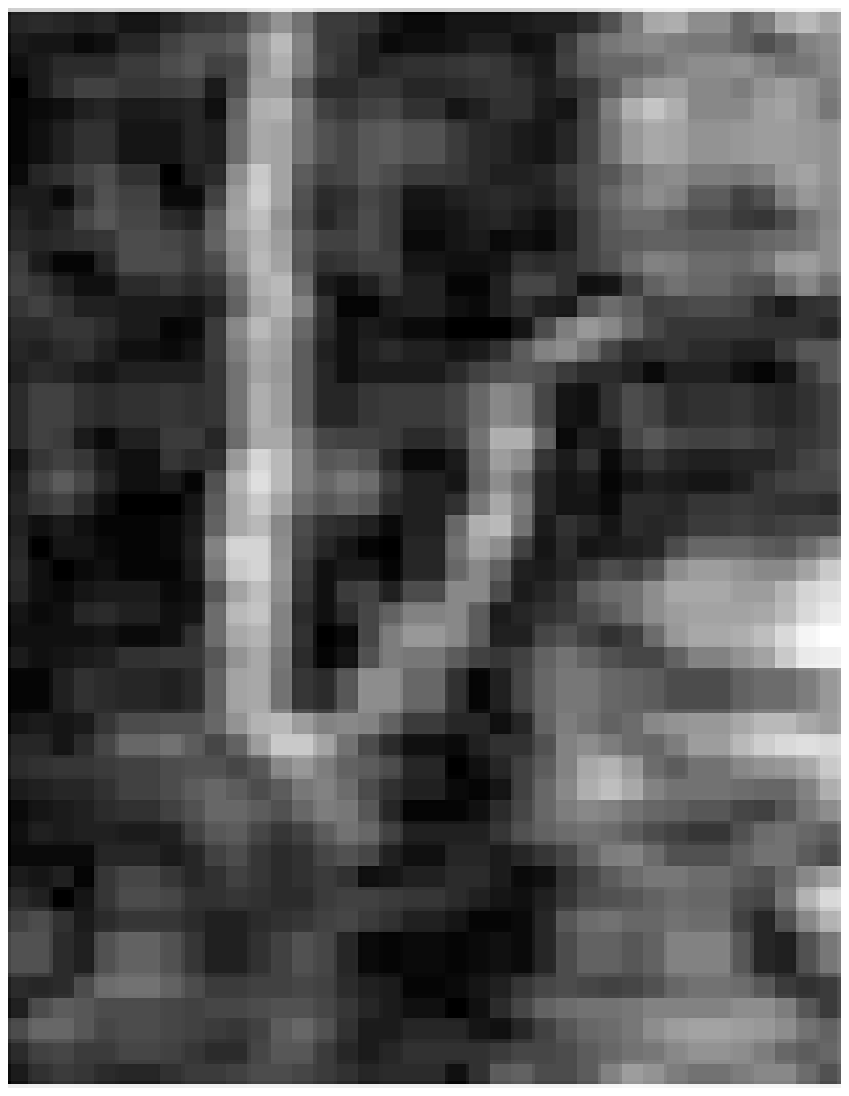}
	\includegraphics[width=0.17 \textwidth]{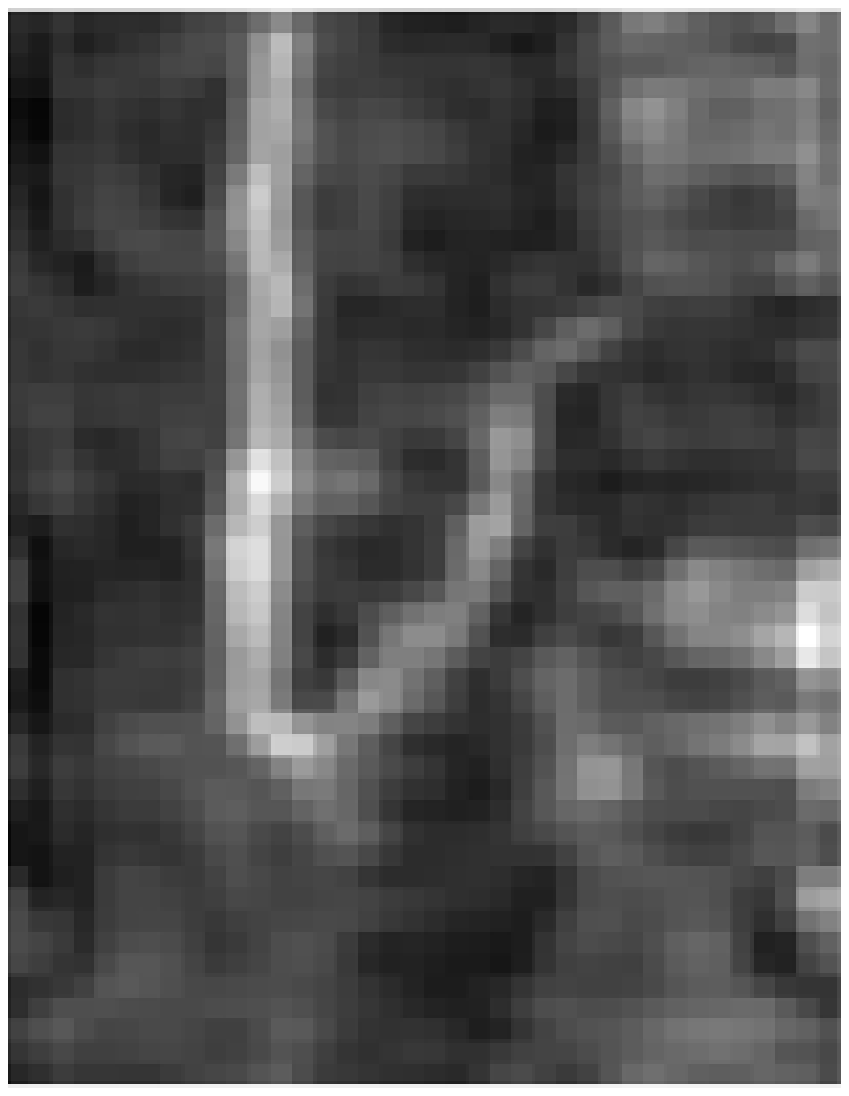}
	\includegraphics[width=0.17 \textwidth]{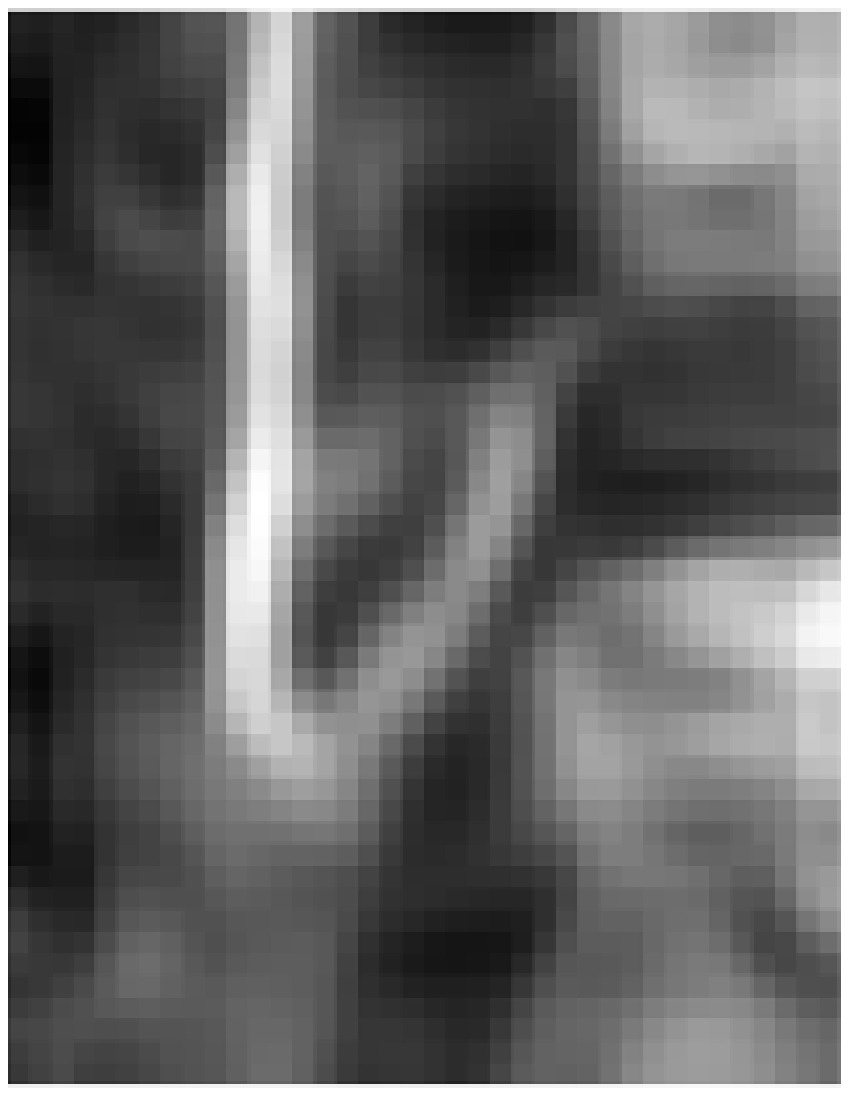}
	\includegraphics[width=0.17 \textwidth]{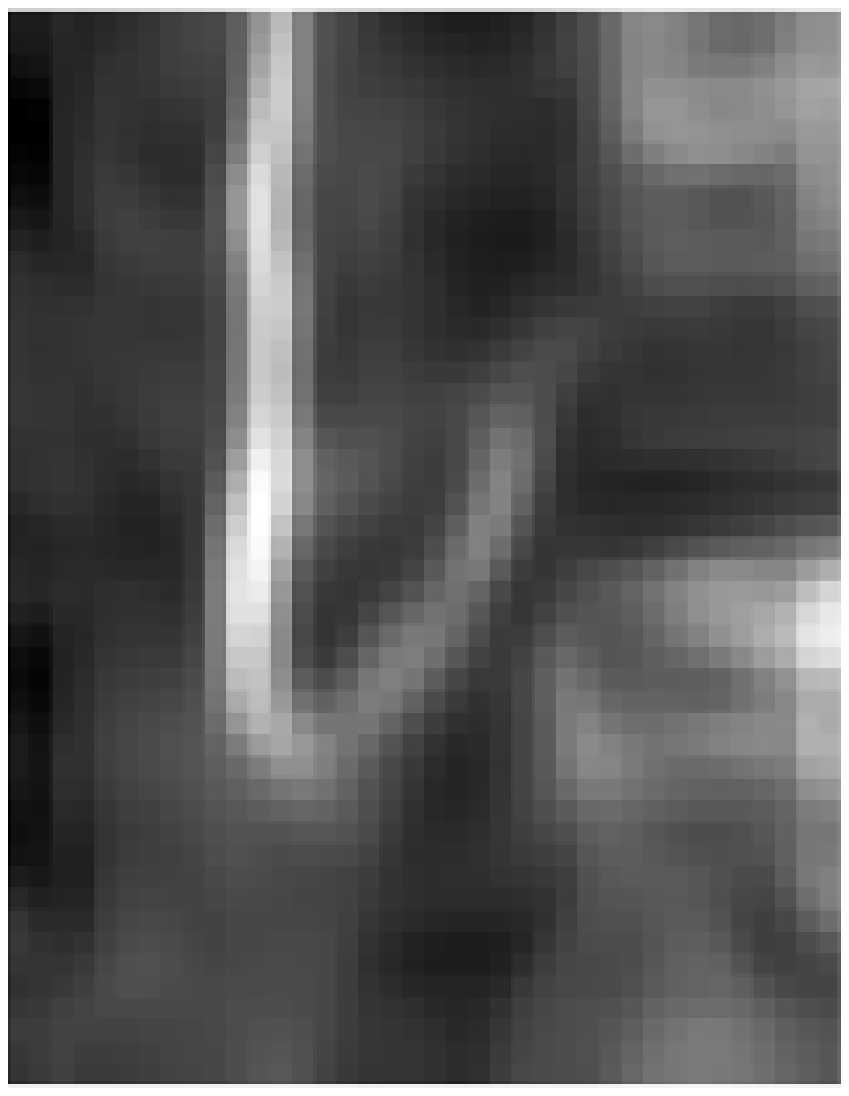}
  \caption{Adaptive Crossing Preserving Flows combined with soft thresholding $\Phi(U)(\vx,\vn)=|U(\vx,\vn)|^{1.5} \sgn(U(\vx,\vn))$ on data containing the Adam Kiewitzc vessel. From left to right: Slice of data, data after soft thresholding, data after CEDOS, data after CEDOS followed by soft thresholding. For parameters see Fig.\ref{fig:Results}, but now $t=5$.}
	\label{fig:ResultsAdamKiewitzc}\end{figure}

\subsection{3D Vessel Tracking in Magnetic Resonance Angiography (MRA) Data}
We use the 3D orientation scores to extend the earlier work on 2D vessel segmentation via invertible orientation scores \cite{Bekkers2013} to 3D vessel segmentation in MRA-data. Even though true crossing structures hardly appear in 3D data, we do encounter vessels touching other vessels/structures. The orientation scores also allow us to better handle complex structures, such as bifurcations. In Fig. \ref{fig:ResultsVesselTracking} we show some first results of the vessel segmentation algorithm.


	\begin{figure}[ht!]
\parbox{.5\textwidth}{
	\includegraphics[width=0.22 \textwidth]{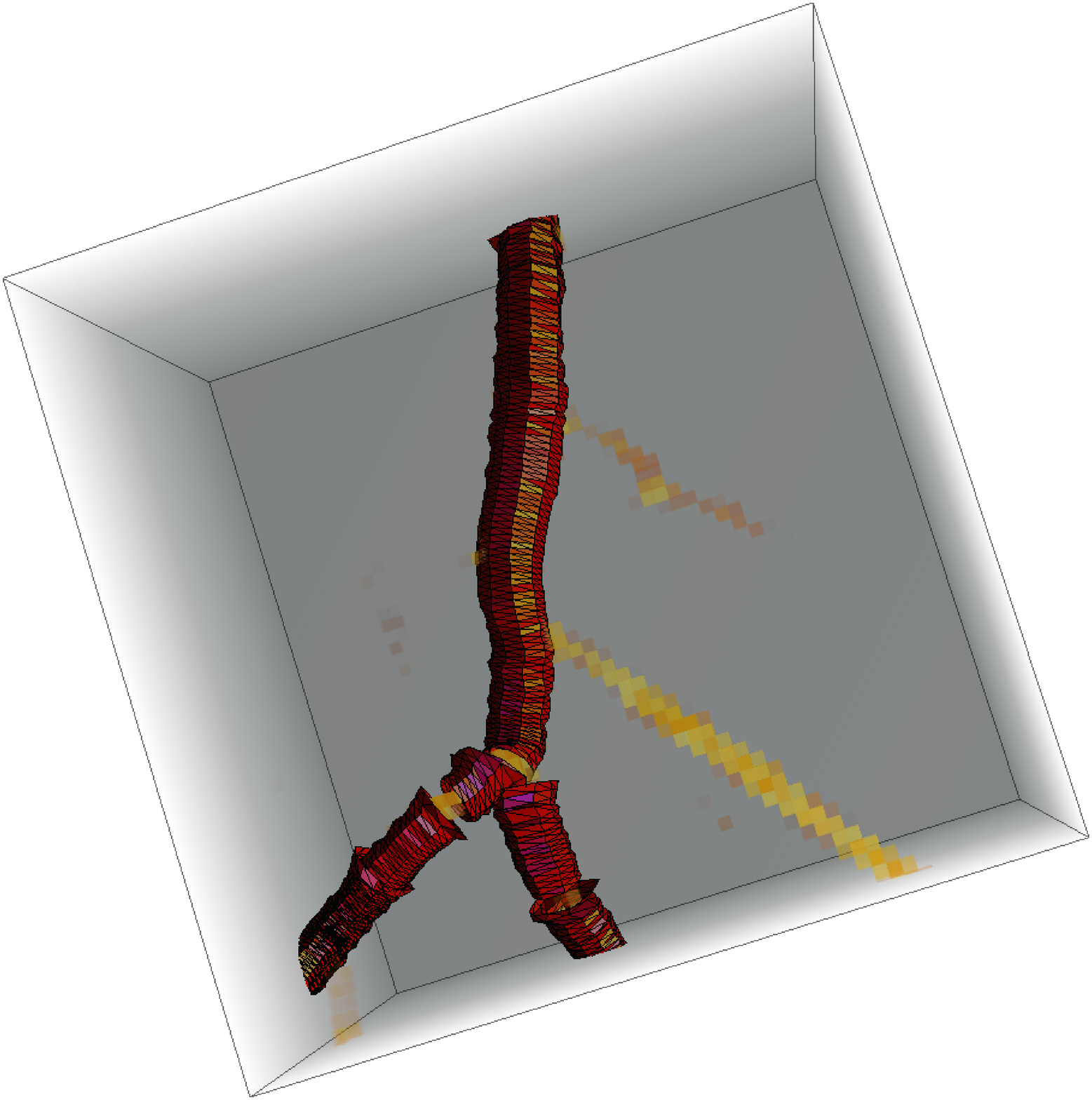}
	\includegraphics[width=0.22 \textwidth]{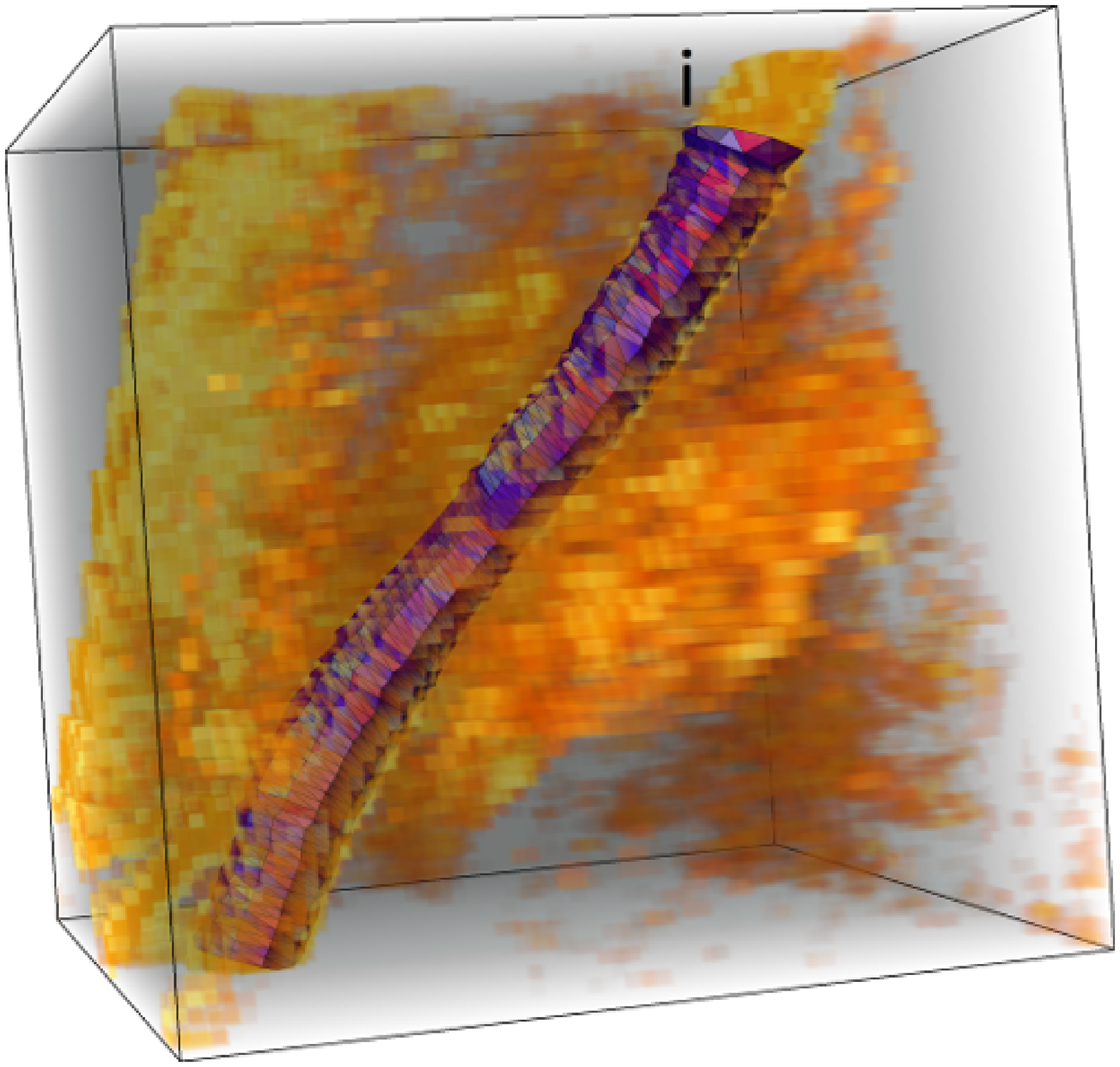}
}\hfill
\parbox{.5\textwidth}{
	\caption{MRA vessel segmentation via invertible orientation scores.}
	\label{fig:ResultsVesselTracking}
}\hfill\null
\end{figure}

\begin{samepage}

\section{Conclusion}
We have extended 2D cake-wavelets to 3D cake-wavelets, which can be used for a 3D invertible orientation score transformation. Efficient implementation for calculating the wavelets via spherical harmonics were introduced. The developed transformation allows us to consider all kinds of enhancement operations via orientation scores such as the adaptive crossing preserving flows which we are currently working on. Next to data-enhancement we also showed some first results of 3D vessel segmentation using 3D orientation scores.

\subsubsection*{Acknowledgements.}
\small
We thank Dr. A.J.E.M. Janssen for advice on the presentation of this paper. The research leading to these results has received funding from the European
Research Council under the European Community's Seventh Framework Programme
(FP7/2007-2013) / ERC grant \emph{Lie Analysis}, agr.~nr.~335555.
\end{samepage}

\appendix

\section{Invertible Orientation Scores of 3D-images and
Continuous Wavelet Theory}\label{app:AppendixA}

The continuous wavelet transform constructed by unitary irreducible representations of locally compact groups was first formulated by Grossman et al.~\cite{Grossmann1985}. Given a Hilbert space $H$ and a unitary irreducible representation $g\mapsto \mathcal{U}_g$ of any locally compact group $G$ in $H$, a non-zero vector $\psi\in H$ is called admissible if
\begin{align}\label{CpsiDef}
C_{\psi}:= \int_{G}\frac{|(\mathcal{U}_{g}\psi,\psi)|^2}{(\psi,\psi)_{H}}d\mu_{G}(g)<\infty,
\end{align}
where $\mu_G$ denotes the left-invariant Haar measure. Given an admissible vector $\psi$ and a unitary representation of a locally compact group $G$ in $H$, the Coherent State (CS) transform $W_\psi:H\rightarrow \mathbb{L}_{2}(G)$ is given by
$(W_\psi[f])(g)=(\mathcal{U}_{g}\psi,f)_H$. $W_\psi$ is an isometric transform onto a unique closed reproducing kernel space $\mathbb{C}_{K_{\psi}}^{G}$ with $K_{\psi}(g,g')=\frac{1}{C_\psi}(\mathcal{U}_g\psi,\mathcal{U}_{g'}\psi)_{H}$ as an $\mathbb{L}_2$-subspace \cite{Ali1998}.

We distinguish between the isometric wavelet transform $W_{\psi}:\mathbb{L}_2^\varrho(\mathbb{R}^3)\rightarrow\mathbb{L}_2(G)$ and the unitary wavelet transform $\cW_{\psi}:\mathbb{L}_2^\varrho(\mathbb{R}^3)\rightarrow \mathbb{C}_K^G$. We drop the formal requirement of $\mathcal{U}$ being square-integrable and $\psi$  being admissible in the sense of \eqref{CpsiDef}, and replace the requirement by  \eqref{eq:AdmissibilityRequirement}, as it is not strictly needed in many cases. This includes our case of interest $G=SE(3)$ and its left-regular action on $\mathbb{L}_2(\mathbb{R}^3)$ where $\cW_{\psi}$ gives rise to an orientation score $\cW_{\psi}f:\R^{3} \rtimes S^{2} \to \mathbb{C}$
\begin{equation}
\cW_{\psi}f(\ul{x},\ul{n})= \widetilde{\cW_{\psi}f}(\ul{x},\ul{R}_{\ul{n}}),
\end{equation}	
with $\ul{R}_{\ul{n}}$ \emph{any} rotation mapping $\ul{e}_{z}$ onto $\ul{n}$ and $\psi$ symmetric around the $z$-axis. Here the domain is the coupled space of positions and orientations: $\R^{3} \rtimes S^{2}:=SE(3)/(\{\ul{0}\} \times SO(2))$, cf.~\!\cite{DuitsIJCV2010}.

From the general theory of reproducing kernel spaces \cite [Thm 18] {ThesisDuits},\cite{Ali2014} (where one does not even rely on the group structure), it follows that $\cW_\psi: \LL_2^\varrho (\R^3) \rightarrow \mathbb {C}_K ^ {\R ^ 3 \rtimes S^2} $ is unitary, where $\mathbb {C}_K ^ {\R ^ 3 \rtimes S^2} $ denotes the abstract complex reproducing kernel space consisting of functions on ${\R ^ 3 \rtimes S^2} $ with reproducing kernel
\begin{equation}
K_{(\vy,\vn)} (\vy',\vn')=(\mathcal{U}_{(\vy,\mR_\vn)}\psi,\mathcal{U}_{(\vy',\mR_ {\vn'})}\psi)_{\mathbb{L}_{2}(\mathbb{R}^3)},
\label{eq:}
\end{equation}
with left-regular representation $(\vy,\mR)\mapsto \mathcal{U}_{(\vy,\mR)}\psi$ given by $(\mathcal{U}_{(\vy,\mR)}\psi ) (\vx) =\psi(\mR^T (\vx-\vy))$. Now, as the characterization of the inner product on $\mathbb {C}_K ^ {\R ^ 3 \rtimes S^2}$ is awkward \cite {Martens1988}, we provide a basic characterization next via the so-called $M_\psi$ inner product. This is in line with the admissibility conditions in \cite{Fuhr2005}.

\begin{theorem}\label{MPsiRecon}
Let $\psi$ be such that (\ref{eq:AdmissibilityRequirement}) holds. Then $\cW_\psi: \LL_2^\varrho (\R^3) \rightarrow \mathbb {C}_K ^ {\R ^ 3 \rtimes S^2} $ is unitary, and we have
\begin{equation}
(f,g)_{\LL_ 2 (\R ^ 3)} = (\cW_\psi f,\cW_\psi g)_{M_{\psi}},
\end{equation}
where $(\cW_\psi f,\cW_\psi g)_{M_{\psi}}=(\mathcal{T}_{M_{\psi}}[\cW_\psi f],\mathcal{T}_{M_{\psi}}[\cW_\psi g])_{\mathbb{L}_{2}(\R^3\rtimes S^2))}$,
with $[\mathcal{T}_{M_{\psi}}[U]](\vy,\vn):=\mathcal{F}^{-1}\bigg{[}\boldsymbol{\omega}\mapsto(2\pi)^{-3/4}M_{\psi}^{-1/2}(\boldsymbol{\omega})\mathcal{F}[U (\cdot,\vn)](\boldsymbol{\omega}) \bigg{]}(\vy)$.
\end{theorem}

\begin{proof}
We rely on \cite [Thm 1] {Sharma2014}, where we set $H =\LL_2 (\R^3)$. The rest follows by well posed restriction to the quotient $\R^3\rtimes S^2$.
\end{proof}

\begin{corollary}
Let $M_\psi>0$ on $\R^3$. The space $\mathbb{C}_{K}^{\R^3\rtimes S^2}$ is a closed subspace of Hilbert space $\mathbb{H}_{\psi}\otimes\mathbb{L}_2(S^2)$, where $\mathbb{H}_{\psi}=\{f\in \mathbb{L}_{2}(\R^3)|\ M_{\psi}^{-\frac{1}{2}}\mathcal{F}[f]\in\mathbb{L}_{2}(\mathbb{R}^3)\}$, and projection of embedding space onto the space of orientation scores is given by
$(\mathbb {P}_\psi (U)) (\vy,\vn)= (K_{(\vn,\vy)},U)_{M_\psi} = (\cW_\psi  \cW_\psi ^ {*,ext} (U)) (\vy,\vn)$,
where $\cW_\psi ^ {*,ext}$ is the natural extension of the adjoint to the embedding space.
\end{corollary}

\bibliographystyle{plain}

\begin{thebibliography}{10}

\bibitem{Ali1998}
S.T. Ali.
\newblock {A general theorem on square-integrability: Vector coherent states}.
\newblock {\em J. Math. Phys.}, 39(8):3954, 1998.

\bibitem{Ali2014}
S.T. Ali, J.-P. Antoine, and J.-P. Gazeau.
\newblock {\em {Coherent states, wavelets, and their generalizations}}.
\newblock Springer, 2014.

\bibitem{Bekkers2013}
E.~Bekkers and R.~Duits.
\newblock {A multi-orientation analysis approach to retinal vessel tracking}.
\newblock {\em JMIV}, 2014.

\bibitem{BurgethBook2009}
B.~Burgeth, S.~Didas, and J.~Weickert.
\newblock {A general structure tensor concept and coherence-enhancing diffusion
  filtering for matrix fields}.
\newblock In {\em Visualization and processing of tensor fields}, pages
  305--324. Springer,Berlin, 2009.

\bibitem{Burgeth2012}
B.~Burgeth, L.~Pizarro, S.~Didas, and J.~Weickert.
\newblock {3D-Coherence-enhancing diffusion filtering for matrix fields}.
\newblock In {\em Mathematical methods for signal and image analysis and
  representation}, pages 49--63. Springer London, 2012.

\bibitem{Creusen2012}
E.J. Creusen, R.~Duits, and T.C.J. Dela~Haije.
\newblock {Numerical schemes for linear and non-linear enhancement of DW-MRI}.
\newblock {\em SSVM}, 1:14--25, 2012.

\bibitem{Descoteaux2007}
M.~Descoteaux, E.~Angelino, S.~Fitzgibbons, and R.~Deriche.
\newblock {Regularized, fast, and robust analytical Q-ball imaging.}
\newblock {\em MRM}, 58(3):497--510, September 2007.

\bibitem{ThesisDuits}
R.~Duits.
\newblock {\em {Perceptual organization in image analysis}}.
\newblock PhD thesis, Technische Universiteit Eindhoven, 2005.

\bibitem{DuitsIJCV2010}
R.~Duits and E.M. Franken.
\newblock {Left-invariant diffusions on the space of positions and orientations
  and their application to crossing-preserving smoothing of HARDI images}.
\newblock {\em IJCV}, 92(3):231--264, March 2010.

\bibitem{GaugeFrameNew}
R.~Duits, M.H.J. Janssen, J.~Hannink, and G.R. Sanguinetti.
\newblock {Locally Adaptive Frames in the Roto-Translation Group and their
  Applications in Medical Imaging}.
\newblock {\em arXiv preprint: arXiv:1502.08002}.

\bibitem{Franken2009}
E.M. Franken and R.~Duits.
\newblock {Crossing-preserving coherence-enhancing diffusion on invertible
  orientation scores}.
\newblock {\em IJCV}, 85(3):253--278, February 2009.

\bibitem{Fuhr2005}
F.~F\"{u}hr.
\newblock {Abstract harmonic analysis of continuous wavelet transforms}.
\newblock {\em Lecture Notes in Mathematics}, 1863, 2005.

\bibitem{Grossmann1985}
A.~Grossmann, J.~Morlet, and T.~Paul.
\newblock {Transforms associated to square integrable group representations. I.
  General results}.
\newblock {\em J. Math. Phys.}, 26(10):2473, 1985.

\bibitem{Kalitzin1999}
S.N. Kalitzin, B.M. {ter Haar Romeny}, and M.A. Viergever.
\newblock {Invertible apertured orientation filters in image analysis}.
\newblock {\em IJCV}, 31:145--158, 1999.

\bibitem{Lee1996}
Tai~Sing Lee.
\newblock {Image representation using 2D Gabor wavelets}.
\newblock {\em IEEE TPAMI}, 18(10), 1996.

\bibitem{Martens1988}
F.J.L. Martens.
\newblock {\em {Spaces of analytical functions on inductive/projective limits
  of Hilbert spaces}}.
\newblock PhD thesis, Technische Universiteit Eindhoven, 1988.

\bibitem{Sharma2014}
U.~Sharma and R.~Duits.
\newblock {Left-invariant evolutions of wavelet transforms on the similitude
  group}.
\newblock {\em accepted for publication ACHA, doi:10.1016/j.acha.2014.09.001}.

\bibitem{Weickert1999}
J.~Weickert.
\newblock {Coherence-enhancing diffusion filtering}.
\newblock {\em IJCV}, 31:111--127, 1999.

\end{thebibliography}

%
%
%
%
%
\end{document}